\documentclass[12pt,leqno]{amsart}
\usepackage{amsmath,amsthm}
\usepackage{amsfonts}
 \newtheorem{res}{Result}[section]

 \newtheorem{theorem}[res]{Theorem}
\newtheorem{remark}[res]{Remark}
 
 \newtheorem{lemma}[res]{Lemma}
 \newtheorem{corollary}[res]{Corollary}

\newtheorem{example}[res]{Example}
\numberwithin{equation}{section}

\def\N{{\mathbb N}}
\def\Z{{\mathbb Z}}

\def\F{{\mathcal F}}

\def\R{{\mathbb R}}

\def\E{{\mathbb E}}
\def\P{{\mathbb P}}
\def\hP{\hat\P}

\def\defto{\buildrel {\mathrm{def}}\over =}

\def\tto{\buildrel {t\rightarrow \infty}\over\longrightarrow}

\def\state{S}
\def\cstate{C}
\def\tX{\tilde X}
\def\tP{\tilde P}
\def\hX{\hat X}
\def\hP{\hat P}
\def\augstate{\hat S}

\def\hold{{H}}
\def\ret{{R}}
\def\cq{{q_0}}
\def\Pinf{P^\infty}
\def\Xinf{X^\infty}
\def\qinf{q^\infty}
\def\tg{\tilde {g}}
\def\trho{\tilde {\rho}}
\def\tai{\tau^{(i)}}
\def\taone{\tau^{(1)}}
\def\taze{\tau^{(0)}}
\def\thold{\tilde\hold}
\def\ts{s^\cstate}

\begin{document}
\bibliographystyle{plain}
\title[limiting waiting at zero]{markov chains conditioned never to wait too long at the origin}
\author{Saul Jacka}
\address{Department of Statistics, University of Warwick, Coventry CV4
 7AL, UK} \email{s.d.jacka@warwick.ac.uk}

\begin{abstract}
Motivated by Feller's coin-tossing problem, we consider the problem of
conditioning an irreducible Markov chain never to wait too long at 0. Denoting  by $\tau$ the
first time that the chain, $X$, waits for at least one unit of time at the
origin, we consider conditioning the chain on the event
$(\tau>T)$. We show there is a weak limit as $T\rightarrow \infty$ in the
cases where either the statespace is finite or $X$ is transient. We
give sufficient conditions for the existence of a weak limit in other
cases and show that we have vague convergence to a defective limit if
the time to hit zero has a lighter tail than $\tau$ and $\tau$ is subexponential.

\end{abstract}

\thanks{I am grateful to an anonymous referee for many helpful suggestions on improving the presentation of this paper}
\maketitle \centerline{Submitted: 21 June 2009}

\section{Introduction and notation}
\subsection{Introduction} 
Feller (in Section XIII.7 of \cite{f1}) showed that if $p^{(k)}_n$ is the
probability that there is no run of heads of length $k$  or
more in $n$ tosses of a fair coin, then, for a suitable positive constant $c_k$
$$
p^{(k)}_n\sim c_ks_k^{n+1},
$$
where $s_k$ is the largest real root in (0,1) of the equation
\begin{equation}\label{coin}
x^k-\sum_{j=0}^{k-1}2^{-(j+1)}x^{k-1-j}=0.
\end{equation}

More generally, if the probability of a head is $p=1-q$, then the same
asympotic formula is valid, with equation (\ref{coin}) modified to
become
\begin{equation}\label{coin2}
x^k-q\sum_{j=0}^{k-1}p^{j}x^{k-1-j}=0,
\end{equation}
and $c_k=\frac{s_k-p}{q((k+1)s_k-k)}$.

The continuous-time analogue of this question is to seek the asymptotic
behaviour of the probability that $Y$, a Poisson process with rate
$r$, has no inter-jump time exceeding one unit by time $T$.
It follows, essentially  from Theorem \ref{rec1} that, denoting  by $\tau_Y$ the first time that
$Y$ waits to jump longer than one unit of time,
\begin{equation}\label{coin3}
\P(\tau_Y>t)\sim c_r e^{-{\phi_r} t},
\end{equation}
for a suitable constant $c_r$, where ${\phi_r}=1$ if $r=1$ and otherwise
${\phi_r}$ is the root (other than $r$ itself) of the equation
\begin{equation}\label{coin4}
xe^{-x}=r e^{-r}.
\end{equation}

A natural extension is then to seek the tail behaviour of the distribution of $\tau\equiv \tau_X$,
the first time that a Markov chain, $X$, waits longer than one unit of
time at a distinguished state,
0. In general, there has also been much interest (see \cite{JW}, \cite{JW2}, \cite{JW3}, \cite{Ber}, \cite{Ber2},
 \cite{Ber3}, \cite{Ferr}, \cite{Kes}, \cite{JR}, \cite{JR2}, \cite{JR3}, \cite{JR4}, \cite{Kyp})
in conditioning an evanescent
Markov process $X$ on
its survival time being increasingly large and in seeing whether a weak limit exists.

\goodbreak
\subsection{Notation}
We consider a continuous-time Markov chain $X$
on a countable state\-space
$S$, with a distinguished state $\partial$. We denote $S\setminus
\{\partial\}$ by $C$. For convenience, and without loss of generality, we assume henceforth that
$S=\Z^+$ or $S=\{0,\ldots,n\}$ and
$\partial=0$ so that $\cstate=\N$ or $\cstate=\{1,\ldots, n\}$.

We assume that $X$ is irreducible, and non-explosive. We denote the
transition semi-group of $X$ by $\{P(t); {t\geq 0}\}$ and its $Q$-matrix by $Q=(q_{ij})$ and set $q_i=-q_{ii}$. We define
the process $\tX$ as $X$ killed on first hitting 0 and we shall usually
assume that $\tX$ is also irreducible on $\cstate$. We denote the substochastic
semigroup for $\tX$ by $(\tP(t))_{t\geq 0}$.
We first define the return and departure epochs as follows:
$$S_0=\inf\{t\geq 0:\; X_t=0\},\hbox{ }T_0=\inf\{t\geq 0:\; X_t\neq 0\}$$
and then for $n\geq 1$:
$$T_n=\inf\{t\geq S_n:\; X_t\neq 0\},$$
and 
$$S_n=\inf\{t\geq T_{n-1}:\; X_t= 0\},$$
then we define the successive holding and return times $H^n_{n\geq 0}$
and $R^n_{n\geq 0}$ by
$$H^0=T_0\hbox{ and }R^0=S_0$$
and
$$H^n=T_n-S_n\hbox{ and }R^n=S_{n+1}-T_n\hbox{ for }n\geq 1.$$
Then we define the
current holding time as follows:
$$\hold_t=t-S_n\hbox{ if }S_n\leq t<T_n\hbox{ for some }n.$$
It will be convenient in what follows to define  
$$\hold_t=\emptyset\hbox{ if }X_t\neq 0.$$

We denote the first time that $X$ waits in 0 for time 1 by $\tau$, i.e.
$$\tau=\inf\{t:\; H_t\geq 1\},$$
and
denote the process $X$ killed at time $\tau$ by $\hX$. We denote the 
statespace augmented by the current holding time in 0 by
$\augstate\defto \cstate\cup\{\{0\}\times[0,1)\}$. By a slight abuse of notation, we denote the (substochastic)
Markov chain $(\hX_t, H_t)$ on the statespace
$\augstate$ by $\hX$ also\footnote{Note that if $\hat X_t\neq
0$ then $H_t=\emptyset$ so $(\hat X_t,H_t)= \hat X_t$.}.
The associated semigroup is
denoted $(\hP(t))_{t\geq 0}$.
Throughout the rest of the paper we denote by $\P_i$ the probability on
Skorokhod
pathspace $D(\state,[0,\infty))$, conditional on $\hX_0=i$, and the
corresponding filtration by $(\F_t)_{t\geq 0}$.
Finally, we denote a typical hitting time of 0 from state $i$ by $\tai_0$
and its density by $\rho_i$. We denote the density of a typical return time,   $\ret^1$, by $\rho$.

\subsection{Convergence/decay parameters for evanescent chains}
\bigskip
We recall (see for example \cite{JR}) that, if $X^*$ is a
Markov chain on $\cstate$, with substochastic transition semigroup $P^*$ and $Q$-
matrix $Q^*=(q^*_{ij})_{(i,j)\in C\times C}$, then $X^*$ is said to be evanescent if it is irreducible
and dies with probability one. In that case,
we define
$$
\alpha_{X^*}=\alpha=\inf\{\lambda\geq 0:\,\int_0^\infty P^*_{ij}(t)e^{\lambda t}dt
=\infty\}
$$
for any $i,j\in \cstate$, and (see, for example, Seneta and Vere-Jones
\cite{Se2})  $X^*$ is classified as $\alpha$-recurrent or $\alpha$-transient
depending on whether $\int_0^\infty P^*_{ij}(t)e^{\alpha t}dt
=\infty$ or is finite. Moreover,
$X^*$ is $\alpha$-recurrent if and only if $\int_0^\infty f^*_{ii}(t)e^{\alpha t}dt
=1,$
where $f^*_{ii}$ is the defective density of the first return time to $i$
(starting in $i$).

In the $\alpha$-recurrent case, $X^*$ is
$\alpha$-positive recurrent if $$\int_0^\infty tf^*_{ii}(t)e^{\alpha
t}dt<\infty,$$
otherwise $X^*$ is
$\alpha$-null recurrent.
Defining $q^*_i=-q^*_{ii}$,
it is easy to see that $\alpha<q^*_i$ for all $i\in \N$ and hence
$$0\leq\alpha\leq \inf_iq^*_i.$$

Thus $\alpha$ measures the rate of
decay of transition probabilities (in $\cstate$).
There is a second decay parameter --- $\mu^*$, which
measures the rate of dying.

We define $\tau^*$ as the death time of $X^*$ and we define
$s^*_i(t)=\sum_jP^*_{ij}(t)=\P_i(\tau^*>t)$ and set
$$\mu^*=\inf\{\lambda:\,\int_0^\infty s^*_{i}(t)e^{\lambda t}dt
=\infty\}.$$

Notice that $\mu^*$ is independent of $i$ by the usual irreducibility
argument, moreover, since $1\geq s^*_i(t)\geq P^*_{ii}(t)$ it follows that
$$0\leq \mu^*\leq \alpha^*.$$

Note that in our current setting, we shall take $X^*=\tX$ and write
$\tau^*=\tau_0$, the first hitting time of 0. We shall denote the
rate of hitting 0, which is the death rate for $X^*$, by $\mu^\cstate$
and $\alpha^*$ by $\alpha^\cstate$ and the survival probabilities for $\tX$ as
$\ts$, so that $\ts_i(t)=\P_i(\tau_0>t)$.

\subsection{Doob $h$-transforms}
Recall (see, for example, III.49 of Williams \cite {R+W}) that
we may form the $h$-transform of a substochastic Markovian semigroup on
$\state$, 
$(P(t))_{t\geq 0}$, if $h:S\rightarrow\R^+$ is
$P$-superharmonic (i.e. $[P(t)h](x)\leq h(x)$ for all $x\in \state$ and for all $t\geq 0$).
The $h$-transform of $P$, $P^h$, is specified by its transition kernel which
is given by
$$P^h(x,dy;t)\defto \frac{h(y)}{h(x)}P(x,dy;t),$$
so that if we consider the corresponding substochastic measures on path-space,
$\P_x$ and $\P^h_x$ (conditional on $X_0=x$) then
$$\frac{d\P^h_x}{d\P_x}\big|_{\F_t}=h(X_t)
$$
and $P^h$ forms another
substochastic Markovian semigroup.
If $h$ is actually {\em space-time} $P$-superharmonic then appropriate changes
need to be made to these definitions. In particular, if $h(x,t)=e^{\phi
t}h_x$ then
$$
\frac{d\P^h_x}{d\P_x}\big|_{\F_t}=e^{\phi
t}h_{X_t}.
$$
As shown in \cite{JR}, in general, when a weak limit or a vague limit
exists for the problem of interest, it must be a Doob-$h$-transform
of the original process, with the state augmented
by the current waiting time in state 0 in the case we study here.

\subsection{Main results}
Denoting by $\hP(t)$ the substochastic transition semigroup for $\hX$, we define
$$s_i(t)\defto \P_i(\tau>t)=\hP(i,\augstate;t)\hbox{ for }i\in\augstate.$$
Our first result is
\begin{theorem}\label{trans}
Suppose that $X$ is transient. Denote $\P_i(X\hbox{ never hits 0})$ by
$\beta_i$ and
define $\Delta=\sum_{j\in \cstate}\frac{q_{0,j}\beta_j}{\cq}.$
Set
\begin{equation}\label{trans1}
p_{(0,0)}\defto p_0=\frac{(1-e^{-\cq})\Delta}{e^{-\cq}+(1-e^{-
\cq})\Delta},
\end{equation}

\begin{equation}\label{trans2}
p_{(0,u)}=\frac{1-e^{-\cq (1-u)}}{1-e^{-\cq}}p_0,
\end{equation}
and 
\begin{equation}\label{trans3}
p_i=\beta_i+(1-\beta_i)p_0\hbox{ for }i\in\cstate.
\end{equation}
Then
\begin{equation}\label{translimit}
s_i(t)\tto p_i\hbox{ for all }i\in\augstate.
\end{equation}
Hence, if we condition $X$ on $\tau=\infty$ we obtain a new Markov
process, $X^\infty$, on $\augstate$  with honest semigroup $\Pinf$ given
by
\begin{equation}\label{cond1}
\Pinf_{i,j}(t)=\frac{p_j}{p_i}\hP_{i,j}(t)\hbox{ for }j\in \cstate
\end{equation}
and
\begin{equation}\label{cond2}
\Pinf_i((0,du);t)=\frac{p_{(0,u)}}{p_i}\hP_i((0,du);t),
\end{equation}
so that
$\Xinf$ looks like a Markov chain with $Q$-matrix given by
$\qinf_{i,j}=\frac{p_j}{p_i}q_{i,j}$ on $\cstate$, whilst $\Xinf$ has a
holding time in 0 with density $d$ given by
$$d(t)=\frac{e^{-\cq t}}{\int_0^1 e^{-\cq s}ds}1_{(t<1)}$$
and a tine-homogeneous jump probability out of state 0 to state $j$ of
$\frac{q_{0,j}p_j}{\cq p_0}$.
\end{theorem}

In the case where $X$ is recurrent, it is clear that
$s_i(t)\tto 0$ for each $i\in\augstate$.

Now let $W\defto \hold^1+\ret^1$ (so
that $W$ is the first return time of $X$ to 0 from 0) and let $g$ be the
(defective)
density of $W 1_{(\hold^1<1)}$ on $(0,\infty)$.
Our first result under these conditions is as follows. It is a
generalisation to our more complex setting of Seneta and Vere-Jones'
result in the $\alpha$-positive case.

\begin{theorem}\label{rec1}
Let
$$I(\lambda)\defto \int_0^\infty e^{\lambda t}g(t)dt=\E e^{\lambda W}1_{(\hold^1<1)},$$
then if
\begin{equation}\label{condn2}
\hbox{there exists a $\phi$ such that }I(\phi)=1,\hbox{ and }I'(\phi-
)<\infty,
\end{equation}
then for each $i\in\augstate$,
$$e^{\phi t}s_i(t)\tto p_i>0,$$
where the function $p$ is given by
\begin{equation}\label{p1}
p_{(0,0)}=\frac{e^{\phi-\cq}}{\phi I'(\phi-)}\defto \kappa,
\end{equation}
\begin{equation}\label{p2}
p_{(0,u)}=\frac{\int_0^{1-u}e^{(\phi-\cq)s}ds}{\int_0^{1}e^{(\phi-
\cq)s}ds}\kappa,
\end{equation}
and, for $i\neq 0$,
\begin{equation}\label{p3}
p_i=F_{i,0}(\phi)\kappa,
\end{equation}
where
\begin{equation}\label{F}
F_{i,0}(\lambda)\defto \E e^{\lambda \tai_0}=\int_0^\infty e^{\lambda t}\rho_i(t)dt.
\end{equation}
\end{theorem}

The following simple condition ensures that condition
(\ref{condn2}) holds.
\begin{lemma}\label{rec2}
Suppose that $\tX$ is $\alpha$-recurrent, and both $N_0\defto \{i:\,
q_{i,0}>0\}$ and $N^*_0\defto \{i:\, q_{0,i}>0\}$ are finite, then (\ref{condn2}) holds. 
\end{lemma}

\begin{corollary}\label{cor1}
Let $X^T$ denote the chain on $\augstate$ obtained by conditioning $\hX$
on the event $(\tau >T)$, then, if condition (\ref{condn2}) holds,
for each $s>0$, the restriction of the law of $X^T$ to
$\F_s$ converges weakly as $T\rightarrow \infty$ to that of $\Xinf$ restricted to $\F_s$, where
the transition semigroup of $\Xinf$ is given by equations (\ref{cond1})
and  (\ref{cond2}).
\end{corollary}

In the case where $I(\phi)<1$ or $I'(\phi-)=\infty$, Theorems \ref{subexp1},
\ref{subexp2} and \ref{subexp3} (may)  apply,
giving some sufficient conditions for weak or vague convergence to take
place.
In Theorem \ref{subexp4} and Corollary \ref{bd2}, we give an application to
the case of a recurrent birth and death process conditioned not to wait
too long in state 0.
\goodbreak

\section{Proof of the transient and $\alpha$-positive cases}
To prove Theorem \ref{trans} is straightforward.

{\em Proof of Theorem \ref{trans}}
It is trivial to establish the equations
\begin{equation}
p_i=\beta_i+(1-\beta_i)p_0
\end{equation}
and
\begin{equation}
p_0=(1-e^{-\cq})\sum_{j\in \cstate}\frac{q_{0,j}}{\cq}p_i.
\end{equation}
Equations (\ref{trans1})-(\ref{trans3}) follow immediately. Then the
conditioning result follows straightforwardly
\hfill$\square$

\begin{example}
\label{eg1}
We take a transient nearest-neighbour random walk with
reflection at 0 and with up-jump rate of $b$ and down-jump rate of $d$. Note that $1-\beta$
is the minimal positive solution to
$P(t)h=h$ with $h(0)=1$, and that $1-\beta_i=(\frac{d}{b})^i$.
\end{example}

The main tool in the proof of Theorem \ref{rec1} is the Renewal Theorem.

{\em Proof of Theorem \ref{rec1}}
Recall that state $(0,u)$ denotes that the killed chain is at 0 and
its current holding time is $u$.
First note that $s_{(0,0)}$ satisfies the renewal equation
\begin{equation}\label{renew1}
s_{(0,0)}(t)=\bigl(1-\int_0^\infty g(u)du\bigr)1_{(t<1)}+\int_t^\infty g(u)du+\int_0^t g(u)s_{(0,0)}(t-u)du.
\end{equation}
If we define
$$f(t)=e^{\phi t}s_{(0,0)}(t),$$
it follows immediately from (\ref{renew1}) that
\begin{equation}\label{renew2}
f(t)=e^{\phi t}\bigl((1-\int_0^\infty g(u)du)1_{(t<1)}+\int_t^\infty g(u)du\bigr)+\int_0^t \tilde g(u)f(t-u)du,
\end{equation}
where
$\tilde g(t)\defto e^{\phi t}g(t).$ Now, it is easy to check that the conditions of Feller's alternative
formulation of the Renewal Theorem (see XI.1 of \cite{Feller}, p.363) are satisfied, so we conclude that
\begin{equation}\label{lim1}
f(t)\tto \mu^{-1}\int_0^\infty e^{\phi t}((1-\int_0^\infty g(u)du)1_{(t<1)})+\int_t^\infty g(u)du)dt,
\end{equation}
where
$$
\mu=\int_0^\infty t\tg(t)dt=I'(\phi-).
$$
It is trivial to establish, by changing the order of integration, that
$$
\int_0^\infty e^{\phi t}\int_t^\infty g(u)dudt=\int_0^\infty g(u)\int_0^u
e^{\phi t}dt du=\frac{I(\phi)-\int_0^\infty g(u)du}{\phi}=\frac{1-
\int_0^1 \cq e^{-\cq u}du}{\phi}=\frac{e^{-\cq}}{\phi},
$$
and hence (\ref{p1}) follows.

To establish (\ref{p3}), notice that (by conditioning on the time of the first hit
of 0),
$$
s_i(t)=\int_t^\infty \rho_i(u)du+\int_0^t \rho_i(u)s_{(0,0)}(t-u)du,$$
and so, denoting $e^{\phi t}s_i(t)$ by $f_i(t)$, we obtain
$$
f_i(t)=e^{\phi t}\int_t^\infty \rho_i(u)du +\int_0^t \trho_i(u)f(t-u)du,$$
where $\trho_i(t)\defto e^{\phi t}\rho_i(t)$. Now $f$ is continuous and
converges to $\kappa$ so, by the Dominated Convergence Theorem,
$$
\int_0^t \trho_i(u)f(t-u)du\tto \int_0^\infty \kappa \trho_i(u)du=\kappa
F_{i,0}(\phi).
$$
Moreover, since $F_{i,0}(\phi)=\int_0^\infty \trho_i(u)du<\infty$, it follows that
$e^{\phi t}\int_t^\infty \rho_i(u)du \leq \int_t^\infty \trho_i(u)du\tto
0$, and hence
$$f_i(t)\tto \kappa F_{i,0}(\phi)$$
as required.

To establish (\ref{p2}), observe that
$$s_{(0,u)}(t)=e^{-\cq t}1_{(t<1-u)}+\int_0^t\int_0^{(1-u)\wedge v}\cq e^{-\cq v}\rho(w-
v)s_{(0,0)}(t-w)dvdw$$
and hence
$$f_{(0,u)}(t)\defto e^{\phi t}s_{(0,u)}(t)=e^{(\phi-\cq) t}1_{(t<1-u)}+\int_{w=0}^t\int_{v=0}^{(1-u)\wedge w}\cq e^{(\phi-\cq) v}\trho(w-
v)f(t-w)dvdw,$$
and hence, by the Dominated Convergence Theorem,
\begin{displaymath}
\begin{array}{rcl}
f_{(0,u)}(t)&\tto &\kappa \int_0^\infty \int_0^{w\wedge (1-u)}\cq e^{(\phi-
\cq) v}\trho(w-v)dv\\
&=&\kappa\int_0^\infty\trho(t)dt\int_0^{1-u}\cq
e^{(\phi-\cq) v}dv=\kappa
\frac{\int_0^{1-u}e^{(\phi-\cq)s}ds}{\int_0^{1}e^{(\phi-
\cq)s}ds},\\
\end{array}
\end{displaymath}
as required \hfill$\square$
\begin{remark}
Note that the case mentioned in the introduction, where $Y$ is a
Poisson($r$) process and we let $\tau_Y$ be the first time that an
interjump time is one or larger, can be addressed using the proof of Theorem
\ref{rec1}. In this case, if we consider that the chain \lq\lq returns
directly to 0" at each jump time $Y$ then
$$
I(\lambda)=\int_0^1re^{(\lambda-r)t}dt,
$$
and so $\phi$ satisfies $r\frac{e^{\phi-r}-1}{\phi-r}=1$ which
establishes (\ref{coin4}), and
$$
e^{\phi t}\P(\tau>t)\tto \frac{e^{\phi-r}}{\phi I'(\phi-)}=\frac{\phi -
r}{r(\phi -1)},
$$
for $r\neq 1$.
The case $r=1$ gives $\phi=1$ and $c_1=2$.
\end{remark}
Now we give the

{\em Proof of Lemma \ref{rec2}}
It follows from Theorem 3.3.2
of \cite{JR3} that if $N_0$ is finite then $\alpha^\cstate=\mu^\cstate$.
Now since $\tX$ is $\alpha$-recurrent it follows that
$$
\int_0^\infty e^{\lambda t}\tP_{ii}(t)dt<\infty\hbox{ iff
}\lambda<\alpha^\cstate.
$$
Since $\ts_i(t)\geq \tP_{ii}(t)$, it follows that
$$
\int_0^\infty e^{\lambda t}\ts_{i}(t)dt=\infty\hbox{ if
}\lambda\geq\alpha^\cstate.
$$
Conversely,
since $\alpha^\cstate=\mu^\cstate$
we see that
$$
\int_0^\infty e^{\lambda t}\ts_{i}(t)dt<\infty\hbox{ if
}\lambda<\alpha^\cstate,
$$
and so we conclude that
$$
\int_0^\infty e^{\lambda t}\ts_{i}(t)dt<\infty\hbox{ iff
}\lambda<\alpha^\cstate.
$$
Now
$$I(\lambda)=\int_0^\infty e^{\lambda t}g(t)dt=\int_0^1e^{(\lambda-
\cq)t}dt\bigl(\sum_{i\in N^*_0}q_{0,i}F_{i,0}(\lambda)\bigr)=\int_0^1e^{(\lambda-
\cq)t}dt\bigl(\sum_{i\in N^*_0}q_{0,i}(\frac{F_{i,0}(\lambda)-
1}{\lambda})\bigr).
$$
Now $F_{i,0}(\lambda)<\infty$ if and only if $\lambda<\alpha^\cstate$
and so, since $N^*_0$ is finite by assumption,  $I(\lambda)<\infty$ iff $\lambda<\alpha^\cstate$. It now follows
trivially that $\phi<\alpha^\cstate$ and that (\ref{condn2}) is satisfied
\phantom{p}\hfill$\square$

Now we give the

{\em Proof of Corollary \ref{cor1}} This follows immediately from
Theorem \ref{rec1} and Theorem 4.1.1 of
\cite{JR} provided that we can show that $h$, given by $h:(i,t)\mapsto
e^{\phi t}p_i$, is $\hP$-harmonic.
This is easy to check by considering the chain at the epochs when it
leaves and returns to 0, i.e. we show that, defining $\sigma$ as the
first exit time from 0,
$\E_{(0,u)}h(\hX_{t\wedge\sigma},t\wedge\sigma)=h((0,u),0)$ and
$\E_ih(\hX_{t\wedge\tau_0},t\wedge\tau_0)=h(i,0)$ for $i\in \cstate$.
This is sufficient since $\hX$ is non-explosive
\hfill$\square$

\section{The $\alpha$-transient case}
We seek now to consider the $\alpha$-transient case. In particular, we
shall focus on the case where $\phi=0$. This is not so specific as one
might think since one can (at the cost of a slight extra difficulty)
reduce the general case to that where $\phi=0$.

\subsection{Reducing to the case where $\phi=0$}
We discuss briefly how to transform the problem to this case.

The essential technique is to note that if, for any $\lambda\leq \phi$, we $h$-transform $\hP$ using
the space-time $\hP$-superharmonic function $h^\lambda$ given by 
$$
h^\lambda(i,t)=F_{i,0}(\lambda)e^{\lambda t}\hbox{ for }i\in \cstate
$$
and
$$h^\lambda((0,u),t)=(1-I(\lambda)\frac{J^\lambda(u)}{J^\lambda(1)})e^{-(\lambda-\cq)u}e^{\lambda t}\hbox{ for }u\in [0,1),
$$
where
$$J^\lambda(x)\defto \int_0^xe^{(\lambda-\cq )v}dv,$$
then we obtain a new chain $\overline X$ on $\augstate$, with $\phi_{\overline
X}=\phi-\lambda$ and satisfying
$g_{\overline X}(t)=e^{\lambda t}g(t)$, which dies only from state $(0,1-)$.

\begin{proof}
It is a  standard result that $h^\lambda$ is space-time harmonic for $\hP$
off $\{0\}\times [0,1)$, while, since $I(\lambda)<1$, it is easy to see that $h^\lambda$ is
superharmonic on 
$\{0\}\times [0,1)$, by conditioning on the time of first exit from 0.
Now it is easy to check that $\overline X$ dies only from state $(0,1-)$ and dies
on a visit to 0 with
probability $1-I(\lambda)$ so  the result follows
immediately.
\end{proof}
\begin{remark}
Note that, in the $\alpha$-null-recurrent case, where $I(\phi)=1$ but
$I'(\phi -)=\infty$, the transform above produces a null-recurrent $h$-transform 
when $\lambda=\phi$, whereas the transform is still evanescent
in the $\alpha$-transient case.

It will follow from L'H\^opital's Theorem in the 
$\alpha$-transient cases that if $\psi_i$ denotes the density (on (1,$\infty$)) of $\tau$
when starting from state $i$, then, if
$\frac{\psi_i(t-v)}{\psi_j(t)}$ has a limit as $t\rightarrow \infty$ then
it is the common limit of $\frac{s_i(t-v)}{s_j(t)}=\frac{\int_{t-v}^\infty \psi_i(u)du}
{\int_t^\infty \psi_j(u)du}$ and $\frac{h^\phi_i}{h^\phi_j}\frac{s^{h^\phi}_i(t-v)}{s^{h^\phi}_j(t)}=
\frac{\int_{t-v}^\infty e^{\phi u}\psi_i(u)du}
{\int_t^\infty e^{\phi u}\psi_j(u)du}$.

In the $\alpha$-null recurrent case, we see that this is not of much
help. It is not hard to generalise Lemma 3.3.3 of \cite{JR2} to prove
that in this case $(i,t)\mapsto e^{\phi t}h^\phi_i$ is the unique $\hP$-superharmonic
function of the form $e^{\lambda t}k_i$ and so gives the only possible
weak or vague limit.
\end{remark}

\subsection{Heavy and subexponential tails}
All the results quoted in this subsection, apart from the last, are taken from Sigman \cite{Sigman}.

Recall first that a random variable (normally taking values in $\R^+$)
$Z$, with distribution function $F_Z$,
is said to be {\em heavy-tailed}, or to have a heavy tail, if
$$
\frac{\overline{F_Z}(t+s)}{\overline{F_Z}(t)}\tto 1\hbox{ for all }s\geq 0,
$$
where $\overline{F_Z}\defto 1-F_Z$, is the complementary distribution function.

Denoting the $n$-fold convolution of $F_Z$ by $F^{n}_Z$, $Z$ is said to
have a {\em subexponential tail}, or just to be subexponential, if
\begin{equation}\label{subdef}
\frac{\overline{F^{n}_Z}(t)}{\overline{F_Z}(t)}\tto n \hbox{ for all }n,
\end{equation}
and (\ref{subdef}) holds iff
\begin{equation}\label{subdef2}
\limsup_{t\rightarrow \infty}\frac{\overline{F^{n}_Z}(t)}{\overline{F_Z}(t)}\leq n \hbox{ for
some }n\geq 2.
\end{equation}
A subexponential random variable always has a heavy tail.

Two random variables, $X$ and $Y$, are said to have comparable tails, or to be tail equivalent, if
$$
\overline{F_Y}(t)\sim c\overline{F_X}(t)$$
for some $c>0$.

$Y$ is said to have a lighter tail than $X$ if
$$
\frac{\overline{F_Y}(t)}{\overline{F_X}(t)}\tto 0.
$$
\begin{lemma}\label{subexp14}
If $X$ and $Y$ are independent, $Y$ is lighter tailed than $X$ and $X$ has a subexponential tail then
$X+Y$ has a subexponential tail and
$$
\overline{F_{X+Y}}(t)\sim \overline{F_X}(t).
$$
\end{lemma}

\begin{lemma}\label{subexp15}
If $X$ and $Y$ are independent and subexponential and tail-equivalent
with 
$$
\overline{F_Y}(t)\sim c\overline{F_X}(t),$$
then $X+Y$ is subexponential and
$$
\overline{F_{X+Y}}(t)\sim (1+c)\overline{F_X}(t).
$$
\end{lemma}

This generalises to the following random case:
\begin{lemma}\label{subexp16}
Suppose that $X_1,\ldots$ are iid with
common d.f $F$ which is subexponential and $N$ is an independent geometric random variable, then if
$$
S\defto \sum_1^N X_i,
$$
then $S$ is subexponential and
$$
\overline{F_S}(t)\sim (\E N)\overline{F_X}(t).
$$
\end{lemma}

Finally, we have the following
\begin{lemma}\label{mix}
Suppose that $X_1,\ldots$ are independent and tail-equivalent with
$$F_{X_i}\defto F_i,
$$
and $J$ is an independent random variable taking values in $\N$.
Let
$$Y=X_J,$$
(so that $Y$ is a mixture of the $X_i$s)
and denote its distribution function by $F$ (so $F(t)=\sum_{i\in
\N}\P(J=i)F_i(t)$).

Now suppose that
$$\overline{F_{i}}(t)\sim a_i\overline{F_{1}}(t):
$$
if the collection 
$\{\frac{\overline{F_J}(t)}{\overline{F_{1}}(t)};\; t\geq 0\}$ is uniformly
integrable then, 
\begin{equation}\label{subexp9}
\overline F(t)\sim (\E a_J)\overline{F_1}(t).
\end{equation}
In particular, if $J$ is a bounded r.v. then (\ref{subexp9}) holds.
\end{lemma}
\begin{proof}
It follows from the assumptions that
$$
\frac{\overline{F_J}(t)}{\overline{F_{1}}(t)}\tto a_J\hbox{ a.s.}
$$
Thus if the collection is u.i. then convergence is also in $L^1$ and so,
since $\E \overline{F_J}(t)=\overline{F}(t)$, we see that
$$
\frac{\overline{F}(t)}{\overline{F_{1}}(t)}\tto \E a_J.
$$
In particular, if $J\leq n$ a.s. then
$$
\limsup_{t\rightarrow \infty}\frac{\overline{F_J}(t)}{\overline{F_{1}}(t)}\leq
\max_{1\leq i\leq n} a_i \hbox{ a.s.},
$$
and so the collection is indeed u.i.
\end{proof}

\subsection{Results for heavy tails}
Suppose first that $0=\phi<\mu^{\cstate}$.
\begin{theorem}\label{subexp1}
If $0=\phi<\mu^{\cstate}$ and $\tau$ is subexponential, then $\frac{s_i(t-
v)}{s_j(t)}\tto 1$ for all $v\geq 0$ and $s_{(0,u)}(t-v)/s_{(0,0)}(t)\tto \frac{1-e^{-
\cq(1-u)}}{1-e^{-\cq}}.$
\end{theorem}
\begin{proof}
Notice first that, since $\mu^{\cstate}>0$, $\P_i(\tau_0>t)\leq k_i e^{-\mu^{\cstate}
t/2}$, so that $\tau^{(i)}_0$ has a lighter tail than $\tau$ so, by Lemma \ref{subexp14},
$$s_i(t-v)=\P_{(0,0)}(\tau^{(i)}_0+\tau>t-v)\sim
\P_{(0,0)}(\tau>t-v)\sim \P_{(0,0)}(\tau>t)=s_{(0,0)}(t).$$
Similarly,
$$s_{(0,u)}(t-v)=\int_0^{1-u}\cq e^{-\cq w}\P(\ret^1+\tau>t-v-w)dw\sim  (1-e^{-
\cq (1-u)})\P(\ret^1+\tau>t)$$
and so $s_{(0,u)}(t-v)/s_{(0,0)}(t)$ converges
to the desired limit.
\end{proof}
It is easy to see that $h$, defined by $h_i=1$, for $i\in C$ and
$h_{(0,u)}=\frac{1-e^{-\cq(1-u)}}{1-e^{-\cq}}$ is strictly $\hP$-
superharmonic and is harmonic on $\cstate$:
the following theorem then follows easily from a mild
adaptation of Theorem 4.1.1 of \cite{JR}.
\begin{theorem}\label{subexp2}
Under the conditions of Theorem \ref{subexp1}, the restriction of the law of $\tX^T$ to
$\F_s$ converges {\em vaguely} to that of $\Xinf$ restricted to $\F_s$, where
$\Pinf$ is the (substochastic) $h$-transform of $\tP$ (which dies from state $(0,u)$ with hazard rate
$\lambda(u)=\frac{\cq e^{-\cq}}{1-e^{-\cq(1-u)}}$).
\end{theorem}

\begin{example}
Consider the case where $\sum_{j\in \cstate}q_{0,j}F_{j,0}(\lambda)=\infty$ for all
$\lambda>0$ but $\mu^{\cstate}>0$. For example, we may take the nearest-neighbour random
walk on $\N$ with up-jump rate $b$ and down-jump rate $d$ (with $b<d$) and then set
$$q_0=1;\; q_{0,i}=\frac{6}{\pi^2i^2}\hbox{ for }i\in \N.
$$
It is well-known that
$$\mu^\cstate=b+d-2\sqrt{bd},$$
and
$$
F_{i,0}(\lambda)=\gamma_\lambda^i,
$$
where
$$\gamma_\lambda=\frac{b+d-\lambda-\sqrt{(b+d-\lambda)^2-4bd}}{2b}>1
$$
for $0<\lambda\leq\mu^\cstate$.
So, for any $\lambda>0$, $\sum_{i\in \N}q_{0,i}F_{i,0}(\lambda)=\E e^{\lambda \ret^1}=\infty$
and hence $\phi=0$.
\end{example}

Now we consider the case where $\mu^{\cstate}=0$ (and hence $\phi=0$ also).

\begin{theorem}\label{subexp3}
Denote by $\tau^{(i)}$ a generic random variable having the distribution
of the $\tau$ conditional on $X_0=i$.
Suppose that $\tau^{(i)}$ have comparable heavy tails, so that
$\P(\tai>t)=\P_i(\tau>t)\sim c_i\P(\taze>t)=c_i\P_{(0,0)}(\tau>t)$ and $\frac{\P_i(\tau>t+s)}{\P_i(\tau>t)}\tto
1$,
then,
defining
$$h_i=c_i\hbox{,  for }i\in \state$$
and
$$
h_{(0,u)}=\frac{1-e^{-\cq(1-u)}}{1-e^{-\cq}},
$$

\begin{equation}\label{subexp22}
\frac{s_j(t-
v)}{s_i(t)}\tto \frac{c_j}{c_i}\hbox{ for all }v\geq 0\hbox{ and for
all }i,j\in\augstate.
\end{equation}
In particular, if the $\tai_0$'s have comparable subexponential tails, with
$$
\P(\tai_0 >t)=\P_i(\tau_0>t)\sim a_i\P(\taone_0>t)=a_i\P_1(\tau_0>t)
$$
and
$$
q_{0,i}=0\hbox{ for }i> n,
$$
then,
defining $a_0=0$,
$m=\sum_{i\in\cstate}q_{0,i}a_i/\cq$,
$$h_i=1+\frac{a_i}{(e^{\cq}-1)m}\hbox{,  for }i\in \state$$
and
$$
h_{(0,u)}=\frac{1-e^{-\cq(1-u)}}{1-e^{-\cq}},
$$
we have that
$$\frac{s_j(t-v)}{s_i(t)}\tto \frac{h_j}{h_i}
\hbox{ for all }v\geq 0\hbox{ and for all }i,j\in\augstate.
$$

In general $a$ must be $\tP$-superharmonic. If $a$ is $\tP$-harmonic then $h$ is $\hP$-harmonic, so that,
in this case, the restriction of the law of $\tX^T$ to
$\F_s$ converges weakly to that of $\Xinf$ restricted to $\F_s$, where
$\Pinf$ is the (stochastic) $h$-transform of $\tP$.
\end{theorem}
\begin{proof}
The first claim is essentially a restatement of the conditions for
convergence in (\ref{subexp22}).

To prove the second statement, first notice that we may write
$$
\tai=\tai_0+1+\sum_{n=1}^N (\thold^n+R^n),
$$
where $(\thold^n)_{n\geq 1}$ are a sequence of iid random variables with
distribution that of the holding time in 0 conditioned on its lying in
(0,1), $N$ is a Geometric($e^{-\cq}$) r.v. and the $R^n$'s are as in
section 2 and all are independent.

Now each $R^n$ is a mixture of $\tai_0$s, so, by Lemma \ref{mix},
$$
\P(R^n\geq t)=\sum_{i\in \cstate} \frac{q_{0,i}}{\cq}\P(\tai_0\geq t)\sim
\sum_{i\in \cstate} \frac{q_{0,i}}{\cq}a_i\P(\taone_0\geq
t)=m\P(\taone_0\geq t).
$$

Now it follows from Lemma \ref{subexp14} that $(\thold^n+R^n)$
is tail equivalent to $R^n$ and is subexponential and then we deduce,
from Lemmas \ref{subexp15} and \ref{subexp16} that $\P(\tai>t)\sim
(a_i+m(e^{\cq}-1))\P(\taone_0\geq t)=m(e^{\cq}-1)h_i\P(\taone_0\geq t).$
The last statement follows from the fact that $\tX$ is non-explosive and
it is then easy to check (by considering the chain at the epochs when it leaves and returns to 0) that
then $h$ is $\hP$-harmonic if $a$ is $\tP$-harmonic
\end{proof}

\begin{theorem}\label{subexp4}
Suppose that $\tX$ is a recurrent birth and death process on $\Z^+$ and, for some
$i$, $\tai_0$ is subexponential, then $\P(\tau^{(j)}_0>t)\sim
\frac{\beta_j}{\beta_i}\P(\tai_0>t)$, where $\beta$ is the unique $\tP$
harmonic function on $\N$ with $\beta_1=1$.
\end{theorem}
\begin{proof}
Notice that, since $\tai_0$ is subexponential, it follows that $\mu^{\cstate}=0$
and hence, by Theorem 5.1.1 of \cite{JR3}, there is a unique $\tP$-
harmonic $\beta$. It follows that for any $n$, $\sigma_n$, the first exit time of $X$
from the set $\{1,\ldots, n-1\}$ has an exponential tail (i.e its tail decreases to 0 at an exponential rate)
and the exit is
to $n$ with probability $\beta_i/ \beta_n$ if $X$ starts in $i$.

It follows that for each $j\leq i$,
$$\P(\tau^{(j)}_0>t)\sim \frac{\beta_j}{\beta_i} \P(\tau^{(i)}_0>t).
$$
Similarly, for $i<n$, $\tai_0=\sigma_n+1_A\tau^{(n)}_0$, where
$A=$($X$ exits $\{1,\ldots,n-1\}$  to $n$), so that
$$\P(\tai_0>t)\sim
\P(A)\P(\tau^{(n)}_0>t)=\frac{\beta_i}{\beta_n}\P(\tau^{(n)}_0>t).
$$
\end{proof}
The following is an immediate consequence of Theorems \ref{subexp3} and
\ref{subexp4}.
\begin{corollary}\label{bd2}
If $\tX$ is a birth and death process on $\Z^+$ and, for some
$i$, $\tai_0$ is subexponential, and for some $n$, $q_{0,j}=0$ for $j>n$ then the conclusion of Theorem
\ref{subexp3} holds.
\end{corollary}
\begin{remark}
If $\mu^{\cstate}=0$ and the process conditioned on not hitting 0 until
time $T$ converges vaguely, then the $\tai_0$'s must have comparable {\em
heavy}
tails. If, in fact the convergence is weak (i.e. to an honest process)
then the vector $a$ must be harmonic for $\tP$.

\end{remark}
\begin{remark}
Suppose that $X$ is a birth and death process, with birth rates $b_i$
equal to the corresponding death rates. If the rates are
decreasing in $i$, then $\tau^{(1)}_0$ is subexponential.

To see this, first observe that, by
conditioning on the first jump, we obtain that
$$\P(\tau^{(1)}_0>t)=\frac{1}{2}\P(E_1>t)+\frac{1}{2}\P(E_1+\tau^{(2)}_0>t),
$$
where $E_1$ is the first waiting time in state $1$.
Now, since
$$\tau^{(2)}_0=\tau^{(2)}_1+\tau^{(1)}_0
$$
and
since $\tau^{(2)}_1$ stochastically dominates $\tau^{(1)}_0$,
we obtain the desired result that
$$
\limsup_{t\rightarrow \infty}\frac{\overline{F^{(2)}}(t)}{\overline F(t)}\leq 2,
$$
where $F$ is the distribution function of $\tau^{(1)}_0$. The result now
follows by (\ref{subdef2}).
\end{remark}
\section{Some concluding remarks}
Sigman \cite{Sigman} gives some conditions which ensure that a random
variable has a subexponential tail.

Many obvious examples exist of the $\alpha$-recurrent case. We have
exhibited a few examples in the $\alpha$-transient case always assuming that $\cstate$
is irreducible. If it is not, then in principle we can divide $\cstate$
into communicating classes $\{\cstate_l:\; l\in L\}$, where $L$ is some countable or finite index set.
It is easy to show
that
$$\phi\leq \inf_{l\in L} \mu^{\cstate_l}.
$$
By adapting the proof of Theorem \ref{subexp1}, it is easy to see that
if $\tau$ is subexponential but $\mu^{\cstate_l}>0$ for some $l\in L$,
then $\frac{s_i(t-v)}{s_j(t)}\tto 1$ for $i,j\in
\cstate_l\cup\{\{0\}\times [0,1)$ and so, as in Theorem \ref{subexp2},
weak convergence of the conditioned chains is not
possible if each $\mu^{\cstate_l}>0$.
Conversely, if $\min\limits_{l\in L} \mu^{\cstate_l}=\mu^{\cstate_{l^*}}$ and $X$ restricted to
$\cstate_{l^*}$ is $\alpha$-recurrent then $\phi=\mu^{\cstate_{l^*}}$
and a suitably adapted version of Theorem \ref{rec1} and Corollary
\ref{cor1} will apply.


\begin{thebibliography}{99}
\bibitem{Ber}R A  Doney and J Bertoin: Some asymptotic results for transient random walks.
{\it Adv. App. Prob.} 28, 207--226  (1996).
\bibitem{Ber2}R A Doney and J Bertoin: Spitzer's condition for random walks and
L\'evy processes. {\it Ann. Inst. H. Poincar\'e} 33, 167--178 (1997).
\bibitem{Ber3}R A Doney and L Chaumont: On L\'evy processes conditioned to stay positive (2005),
to appear in {\it JLMS}.
\bibitem{f1}W Feller: \lq\lq An Introduction To Probability Theory and its Applications Vol I'', 3rd Edn. (revised printing).
Wiley, New York (1970).
\bibitem{Feller}W Feller: \lq\lq An Introduction To Probability Theory and
its Application Vol II''. 2nd Edition. Wiley New York (1971)
\bibitem{Ferr} P A Ferrari, H Kesten, S Martinez, and P Picco: Existence of
quasi stationary distributions. A renewal dynamical approach.
{\it Ann. Probab.}, 23 (2), 501--521 (1995)
\bibitem{JW}Saul Jacka, Zorana Lazic and Jon Warren: Conditioning an additive functional of 
a Markov chain to stay non-negative I: survival for a long time. 
{\it Adv. Appl. Prob.} 37 (4), 1015--1034, (2005)
\bibitem{JW2}Saul Jacka, Zorana Lazic and Jon Warren:
Conditioning an additive functional of a Markov chain to stay non-negative II:
hitting a high level. {\it Adv. Appl. Prob.} 37 (4), 1035--1055 (2005)
\bibitem{JR}S D Jacka and G O Roberts: Strong forms of weak convergence. {\it Stoch. Proc. \& Appl.}
67(1), 41--53  (1994).
\bibitem{JR3}S D Jacka and G O Roberts: Weak convergence of conditioned
processes on a countable state space. {\it J. Appl. Probab.}
32(4), 902--916 (1995).
\bibitem{JW3}Saul Jacka and Jon Warren: Examples of convergence and non-convergence
of Markov chains conditioned not to die'. {\it Elec. J. Prob.} 7 (1) (2002)
\bibitem{Kes}H Kesten: A ratio limit theorem for (sub) Markov chains on
$\{1,2,..\}$ with bounded jumps. {\it Adv. Appl. Prob.}, 27 (3), 652--691 (1995)
\bibitem{Kyp}Andreas E. Kyprianou and Zbigniew Palmowski: Quasi-stationary distributions for L\'evy processes.
{\it Bernoulli} 12 (4), 571--581 (2006).
\bibitem{JR4}G O Roberts and S D Jacka: Weak convergence of conditioned birth and death processes
{\it J. Appl. Prob.} 31 (1), 90--100 (1994).
\bibitem{JR2}G O Roberts,  S D Jacka and P K Pollett: Non-explosivity
of limits of conditioned birth and death processes. {\it J. Appl. Probab.} 34(1), 35--
45 (1997).
\bibitem{Se1}E Seneta:\lq\lq Non-negative Matrices and Markov Chains",
Springer-Verlag, New York, Heidelberg, Berlin (1981).
\bibitem{Se2}E Seneta and D Vere-Jones: On quasi-stationary distributions
in discrete time Markov chains with a denumerable infinity of states.
{\it J. Appl. Prob.} {\bf 3}, 403--434 (1966).
\bibitem{Sigman}K Sigman: Appendix: A primer on heavy-tailed distributions {\it Queueing Systems}, 33, 261--275 (1999).
\bibitem{R+W}D Williams:
\lq\lq Diffusions, Markov processes, and martingales: Vol. I''. 
Wiley, New York (1979).

 






\end{thebibliography}
\end{document}